\documentclass[12pt,reqno]{amsart}

\usepackage[latin1]{inputenc}
\usepackage{amsmath}
\usepackage{amsfonts}
\usepackage{amssymb}
\usepackage{graphics}
\usepackage{enumerate}
\usepackage{amssymb,amsmath,amsthm,amscd,latexsym,verbatim,graphicx,amsfonts,float}

\usepackage{amscd}
\usepackage{amsmath}
\usepackage{amssymb}
\usepackage{amsthm}
\usepackage{latexsym}
\usepackage{verbatim}

\theoremstyle{plain}
\newtheorem{theorem}{Theorem}[section]

\newtheorem{corollary}[theorem]{Corollary}

\newtheorem{prop}[theorem]{Proposition}
\theoremstyle{definition}
\newtheorem{conj}[theorem]{Conjecture}

\newtheorem{pr}[theorem]{Problem}
\theoremstyle{remark}

\newcommand{\bbD}{\mathbb{D}}

\newcommand{\bbN}{\mathbb{N}}

\newcommand{\mcp}{\mathcal{P}}

\title[]{New Perspectives on Torsional Rigidity and Polynomial Approximations of z-bar}
\author[]{Adam Kraus $\&$ Brian Simanek}
\date{}

\begin{document}
\maketitle

\begin{abstract}
We consider polynomial approximations of $\bar{z}$ to better understand the torsional rigidity of polygons.  Our main focus is on low degree approximations and associated extremal problems that are analogous to P\'olya's conjecture for torsional rigidity of polygons.  We also present some numerics in support of P\'{o}lya's Conjecture on the torsional rigidity of pentagons.
\end{abstract}

\vspace{4mm}


\vspace{2mm}


\vspace{2mm}

\normalsize

\section{Introduction}\label{Intro}

Let $\Omega\subseteq\mathbb{C}$ be a bounded and simply connected region whose boundary is a Jordan curve.  We will study the  \textit{torsional rigidity} of $\Omega$ (denoted $\rho(\Omega)$), which is motivated by engineering problems about a cylindrical beam with cross-section $\Omega$.   One can formulate this quantity mathematically for simply connected regions by the following variational formula of Hadamard type
\begin{align}\label{rhodef}
\rho(\Omega):=\sup_{u\in C^1_0(\bar{\Omega})}\frac{4\left(\int_{\Omega}u(z)dA(z)\right)^2}{\int_{\Omega}|\nabla u(z)|^2dA(z)},
\end{align}
where $dA$ denotes area measure on $\Omega$ and $C^1_0(\bar{\Omega})$ denotes the set of all continuously differentiable functions on $\Omega$ that vanish on the boundary of $\Omega$ (see \cite{PS} and also \cite{BFL,Makai}).  The following basic facts are well known and easy to verify:
\begin{itemize}
\item for any $c\in\mathbb{C}$, $\rho(\Omega+c)=\rho(\Omega)$,
\item for any $r\in\mathbb{C}$, $\rho(r\Omega)=|r|^4\rho(\Omega)$,
\item if $\Omega_1$ and $\Omega_2$ are simply connected and $\Omega_1\subseteq\Omega_2$, then $\rho(\Omega_1)\leq\rho(\Omega_2)$,
\item if $\bbD=\{z:|z|<1\}$, then $\rho(\bbD)=\pi/2$.
\end{itemize}

There are many methods one can use to estimate the torsional rigidity of the region $\Omega$ (see \cite{Mush,Sok}).  For example, one can use the Taylor coefficients for a conformal bijection between the unit disk and the region (see \cite[pages 115 $\&$ 120]{PS} and \cite[Section 81]{Sok}), the Dirichlet spectrum for the region (see \cite[page 106]{PS}), or the expected lifetime of a Brownian Motion (see \cite[Equations 1.8 and 1.11]{BVdbC} and \cite{HMP}).  These methods are difficult to apply in general because the necessary information is rarely available.

More recently, Lundberg et al. proved that since $\Omega$ is simply connected, it holds that
\begin{equation}\label{fzapp}
\rho(\Omega)=\inf_{f\in A^2(\Omega)}\int_{\Omega}|\bar{z}-f|^2dA(z),
\end{equation}
where $A^2(\Omega)\subseteq L^2(\Omega,dA)$ is the Bergman space of $\Omega$ (see \cite{FL17}). The right-hand side of \eqref{fzapp} is the square of the Bergman analytic content of $\Omega$, which is the distance from $\bar{z}$ to $A^2(\Omega)$ in $L^2(\Omega,dA)$.  This formula was subsequently used extensively in \cite{FS19} to calculate the approximate torsional rigidity of various regions.  To understand their calculations, let $\{p_n\}_{n=0}^{\infty}$ be the sequence of Bergman orthonormal polynomials, which are orthonormal in $A^2(\Omega)$.  By \cite[Theorem 2]{Farrell} we know that $\{p_n(z;\Omega)\}_{n\geq0}$ is an orthonormal basis for $A^2(\Omega)$ (because $\Omega$ is a Jordan domain) and hence
\[
\rho(\Omega)= \int_{\Omega}|z|^2dA-\sum_{n=0}^{\infty}|\langle1,wp_n(w)\rangle|^2
\]
(see \cite{FS19}).  Thus, one can approximate $\rho(\Omega)$ by calculating
\[
\rho_N(\Omega):=\int_{\Omega}|z|^2dA-\sum_{n=0}^{N}|\langle1,wp_n(w)\rangle|^2
\]
for some finite $N\in\bbN$.  Let us use $\mcp_n$ to denote the space of polynomials of degree at most $n$.  Notice that $\rho_N(\Omega)$ is the square of the distance from $\bar{z}$ to $\mcp_N$ in $L^2(\Omega,dA)$.  For this reason, and in analogy with the terminology of Bergman analytic content, we shall say that $\mbox{dist}(\bar{z},\mcp_N)$ is the \textit{Bergman $N$-polynomial content} of the region $\Omega$.  The calculation of $\rho_n(\Omega)$ is a manageable task in many applications, as was demonstrated in \cite{FS19}.  One very useful fact is that $\rho_N(\Omega)\geq\rho(\Omega)$, so these approximations are always overestimates (a fact that was also exploited in \cite{FS19}).

Much of the research around torsional rigidity of planar domains focuses on extremal problems and the search for maximizers under various constraints.  For example, Saint-Venant conjectured that among all simply connected Jordan regions with area $1$, the disk has the largest torsional rigidity.  This conjecture has since been proven and is now known as Saint-Venant's inequality (see \cite{Pol}, \cite[page 121]{PS}, and also \cite{BFL,OR}).  It has been conjectured that the $n$-gon with area $1$ having maximal torsional rigidity is the regular $n$-gon (see \cite{Pol}).  This conjecture remains unproven for $n\geq5$.  It was also conjectured in \cite{FS19} that among all right triangles with area $1$, the one that maximizes torsional rigidity is the isosceles right triangle.  This was later proven in a more general form by Solynin in \cite{Sol20}.  Additional results related to optimization of torsional rigidity can be found in \cite{Lipton,SZ,vdBBV}.

The formula \eqref{fzapp} tells us that maximizing $\rho$ within a certain class of Jordan domains means finding a domain on which $\bar{z}$ is not well approximable by analytic functions (see \cite{FK}).  This suggests that the Schwarz function of a curve is a relevant object.  For example, on the real line, $f(z)=z$ satisfies $f(z)=\bar{z}$ and hence we can expect that any region that is always very close to the real line will have small torsional rigidity.  Similar reasoning can be applied to other examples and one can interpret \eqref{fzapp} as a statement relating the torsional rigidity of $\Omega$ to a similarity between $\Omega$ and an analytic curve with a Schwarz function.  Some of the results from \cite{Makai} are easily understood by this reasoning.

The quantities $\rho_N$, defined above, suggest an entirely new class of extremal problems related to torsional rigidity.  In this vein, we formulate the following conjecture, which generalizes P\'{o}lya's conjecture:
\begin{conj}\label{rhon}
For an $n,N\in\bbN$ with $n\geq3$, the convex $n$-gon of area $1$ that maximizes the Bergman $N$-polynomial content is the regular $n$-gon.
\end{conj}
\noindent We will see by example why we need to include convexity in the hypotheses of this conjecture (see Theorem \ref{nomax} below).

The most common approach to proving conjectures of the kind we have mentioned is through symmetrization.  Indeed, one can prove P\'{o}lya's Conjecture and the St. Venant Conjecture through the use of Steiner symmetrization.  This process chooses a  line $\ell$ and then replaces the intersection of $\Omega$ with every perpendicular $\ell'$ to $\ell$ by a line segment contained in $\ell'$, centered on $\ell$, and having length equal to the $1$-dimensional Lebesgue measure of $\ell'\cap\Omega$.  This procedure results in a new region $\Omega'$ with $\rho(\Omega')\geq\rho(\Omega)$.  Applications of this method and other symmetrization methods to torsional rigidity can be found in \cite{Sol20}.

\smallskip

The rest of the paper presents significant evidence in support of Conjecture \ref{rhon} and also the P\'{o}lya Conjecture for $n=5$.  The next section will explain the reasoning behind Conjecture \ref{rhon} by showing that many optimizers of $\rho_N$ exhibit as much symmetry as possible, though we will see that Steiner symmetrization does not effect $\rho_N$ the same way it effects $\rho$.  In Section \ref{pentnum}, we will present numerical evidence in support of P\'{o}lya's Conjecture for pentagons by showing that among all equilateral pentagons with area $1$, the one with maximal torsional rigidity must be very nearly the regular one.

\section{New Conjectures and Results}\label{new}


Let $\Omega$ be a simply connected Jordan region in the complex plane (or the $xy$-plane).  Our first conjecture asserts that there is an important difference between the Bergman analytic content and the Bergman $N$-polynomial content.  We state it as follows.

\begin{conj}\label{nomaxn}
For each $N\in\bbN$, there is an $n\in\bbN$ so that among all $n$-gons with area $1$, $\rho_N$ has no maximizer.
\end{conj}

We will provide evidence for this conjecture by proving the following theorem, which shows why we included the convexity assumption in Conjecture \ref{rhon}.

\begin{theorem}\label{nomax}
Among all hexagons with area $1$, $\rho_1$ and $\rho_2$ have no maximizer.
\end{theorem}

Before we prove this result, let us recall some notation.  We define the moments of area for $\Omega$ as in \cite{FS19} by
\[
I_{m,n}:=\int_\Omega x^my^ndxdy,\qquad\qquad\qquad m,n\in\mathbb{N}_0.
\]
In \cite{FS19} it was shown that if the centroid of $\Omega$ is $0$, then
\begin{equation}\label{rho1form}
\rho_1(\Omega)=4\frac{I_{2,0}I_{0,2}-I_{1,1}^2}{I_{2,0}+I_{0,2}}
\end{equation}
(see also \cite{DW}).

One can write down a similar formula for $\rho_2(\Omega)$, which is the content of the following proposition.

\begin{prop}\label{rho2def}
Let $\Omega$ be a simply connected, bounded region of area 1 in $\mathbb{C}$ whose centroid is at the origin. Then \newline $\rho_2(\Omega)=4\Big(I_ {0,4} I_ {1,1}^2 - 4 I_ {1,1}^4 - 
     2 I_ {0,3} I_ {1,1} I_ {1,2} + I_ {0,2}^3 I_ {2,
       0} + I_ {0,3}^2 I_ {2,0} + 
     4 I_ {1,2}^2 I_ {2
       ,0} - I_ {1,1}^2 I_ {2,0}^2 - I_ {0,
         2}^2 (I_ {1,1}^2 + 2 I_ {2,0}^2) - 
     6 I_ {1,1} I_ {1,2} I_ {2,1} - 
     2 I_ {0,3} I_ {2,0} I_ {2,1} + I_ {2,0} I_ {2,1}^2 + 
     2 I_ {1,1}^2 I_ {2,2} + 2 I_ {0,3} I_ {1,1} I_ {3,0} - 
     2 I_ {1,1} I_ {2,1} I_ {3,0} + 
     I_ {0,2} (4 I_ {2,1}^2 + (I_ {1,2} - I_ {3,0})^2 + 
        I_ {2,0} (-I_ {0,4} + 6 I_ {1,1}^2 + I_ {2,0}^2 - 
           2 I_ {2,2} - I_ {4,0})) + I_ {1,1}^2 I_ {4,
       0}\Big)\Big/\Big((I_ {0,3} + I_ {2,1})^2 + (I_ {1,2} + 
     I_ {3,0})^2 + (I_ {0,2} + I_ {2,0}) (-I_ {0,4} + 
      4 I_ {1,1}^2 + (I_ {0,2} - I_ {2,0})^2 - 2 I_ {2,2} - 
      I_ {4,0})\Big)$
      \end{prop}
\begin{proof}
It has been shown in \cite{FS19} that    
\begin{equation}\label{rho2det}
|\langle 1,wp_2(w)\rangle |^2=
\begin{vmatrix}
c_{0,0}&c_{0,1}&c_{0,2} \\ 
c_{1,0}&c_{1,1}&c_{1,2} \\ 
c_{0,1}&c_{0,2}&c_{0,3}
\end{vmatrix}^2\Bigg/\left(\begin{vmatrix}
c_{0,0}&c_{1,0} \\ 
c_{0,1}&c_{1,1}  \end{vmatrix}\cdot\begin{vmatrix}
c_{0,0}&c_{0,1}&c_{0,2} \\ 
c_{1,0}&c_{1,1}&c_{1,2} \\ 
c_{2,0}&c_{2,1}&c_{2,2}
\end{vmatrix}\right)  \end{equation}
where
\[
c_{m,n}=\langle z^m,z^n\rangle=\int_\Omega z^m\bar z^ndA(z)
\] 

We can then write 
\begin{align*}
\rho_2(\Omega)&=\rho_1(\Omega)-|\langle 1,wp_2(w;\Omega)\rangle|^2\\
&=4\frac{I_{2,0}I_{0,2}-I_{1,1}^2}{I_{2,0}+I_{0,2}}-|\langle 1,wp_2(w;\Omega)\rangle|^2
\end{align*}
If one calculates $|\langle 1,wp_2(w;\Omega)\rangle|^2$ using \eqref{rho2det}, one obtains the desired formula for $\rho_2({\Omega})$. 
\end{proof}

\begin{proof}[Proof of Theorem \ref{nomax}]
We will rely on the formula \eqref{rho1form} in our calculations.  To begin, fix $a>0$ and construct a triangle with vertices $(-\epsilon,0)$, $(\epsilon/2,\frac{\epsilon\sqrt{3}}{2})$, and $(-\epsilon/2,\frac{\epsilon\sqrt{3}}{2})$, where $\epsilon=\frac{2}{3a\sqrt{3}}$. Consider also the set of points $S=\{(-a/2,\frac{a\sqrt{3}}{2}),(a,0),(-a/2,-\frac{a \sqrt{3}}{2})\}$.  To each side of our triangle, append another triangle whose third vertex is in the set $S$, as shown in Figure 1.  Let this resulting ``windmill" shaped region be denoted by $\Gamma_a$.
{
\begin{figure}[H] 
    \centering
    \includegraphics[scale=0.5]{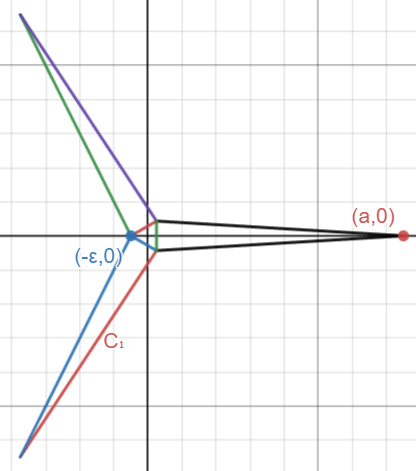}
    \caption{\textit{The region $\Gamma_a$ from the proof of Theorem \ref{nomax}.}}
    \label{label1}
\end{figure}}

To calculate the moments of area, we first determine the equations of the lines that form the boundary of this region. Starting with $C_1(x,a)$ in the 3rd quadrant and moving clockwise we have:
\begin{align*}
C_1(x,a)&=x\sqrt{3}-\frac{6(2x+a)}{2\sqrt{3}+9a^2}\\
C_2(x,a)&=\frac{3a(2+3ax\sqrt{3})}{-4\sqrt{3}+9a^2}\\
C_3(x,a)&=\frac{3a(2+3ax\sqrt{3})}{4\sqrt{3}-9a^2}\\
C_4(x,a)&=-x\sqrt{3}+\frac{6(2x+a)}{2\sqrt{3}+9a^2}\\
C_5(x,a)&=\frac{3(x-a)}{\sqrt{3}-9a^2}\\
C_6(x,a)&=\frac{-3(x-a)}{\sqrt{3}-9a^2}
\end{align*}
To determine $\rho_1(\Gamma_a)$, we calculate the terms $I_{2,0},I_{0,2},$ and $I_{1,1}$ with boundaries determined by the lines given above. Thus for $m,n\in\{0,1,2\}$ we have
\begin{align*}I_{m,n}(\Gamma_a)&=\int_{-\epsilon}^{\epsilon/ 2}\int_{C_1}^{C_4} x^my^n dydx
+
\int_{\epsilon/2}^a\int_{C_6}^{C_5} x^my^n dydx
+
\int_{-a/2}^{-\epsilon}\int_{C_3}^{C_4}x^my^n dydx\\
&\hspace{17mm}+ \int_{-a/2}^{-\epsilon}\int_{C_1}^{C_2}x^my^n dydx
\end{align*}

These are straightforward double integrals and after some simplification, we obtain
\[
\rho_1(\Gamma_a)=4\frac{I_{2,0}I_{0,2}-I_{1,1}^2}{I_{2,0}+I_{0,2}}=\frac{1}{162}\left(3\sqrt{3}+\frac{4}{a^2}+27a^2\right)
\]
and
\[
\rho_2(\Gamma_a)=\frac{1}{1620}\left(3\sqrt{3}+\frac{4}{a^2}+27a^2(1+90/(27a^4-6\sqrt{3}a^2+4))\right),
\]
where we used the formula from Proposition \ref{rho2def} to calculate this last expression.
Notice that we have constructed $\Gamma_a$ so that the area of $\Gamma_a$ is $ 1$ for all $a>0$.  Thus, as $a\rightarrow\infty$, it holds that $\rho_j(\Gamma_a)\rightarrow\infty$ for $j=1,2$.
\end{proof}

Theorem \ref{nomax} has an important corollary, which highlights how the optimization of $\rho_N$ is fundamentally different than the optimization of $\rho$.

\begin{corollary}\label{nosteiner}
Steiner symmetrization need not increase $\rho_1$ or $\rho_2$.
\end{corollary}

\begin{proof}
If we again consider the region $\Gamma_a$ from the proof of Theorem \ref{nomax}, we see that if we Steiner symmetrize this region with respect to the real axis, then the symmetrized version is a thin region that barely deviates from the real axis (as $a$ becomes large).  Thus, $\bar{z}$ is approximately equal to $z$ in this region and one can show that $\rho_1$ of the symmetrized region remains bounded as $a\rightarrow\infty$.  Since $\rho_2\leq\rho_1$, the same holds true for $\rho_2$.
\end{proof}


Our next several theorems will be about triangles.  For convenience, we state the following basic result, which can be verified by direct computation.

\begin{prop}\label{centroidtri}
Let $\Delta$ be the triangle with vertices $(x_1,y_1),(x_2,y_2),(x_3,y_3)$ and centroid $\vec{c}$. Then
\[
\vec{c}=\left(\frac{x_1+x_2+x_3}{3},\frac{y_1+y_2+y_3}{3}\right)
\]
\end{prop}

For the following results, we define the \textit{base} of an isosceles triangle as the side whose adjacent sides have equal length to each other. In the case of an equilateral triangle, any side may be considered as the base.

Here is our first open problem about maximizing $\rho_N$ for certain fixed collections of triangles.

\begin{pr}\label{fixedbasen}
For each $N\in\bbN$ and $a>0$, find the triangle with area $1$ and fixed side length $a$ that maximizes $\rho_N$.
\end{pr}

Given the prevalence of symmetry in the solution to optimization problems, one might be tempted to conjecture that the solution to Problem \ref{fixedbasen} is the isosceles triangle with base $a$.  This turns out to be true for $\rho_1$, but it is only true for $\rho_2$ for some values of $a$.  Indeed, we have the following result.

\begin{theorem}\label{steinertri}
(i) Among all triangles with area 1 and a fixed side of length $a$, the isosceles triangle with base $a$ maximizes $\rho_1$. 

\medskip

(ii) Let $t_*$ be the unique positive root of the polynomial
\[
999x^4/64-93x^3-664x^2-5376x-9216.
\]
If $0<a\leq t_*^{1/4}$, then among all triangles with area 1 and a fixed side of length $a$, the isosceles triangle with base $a$ maximizes $\rho_2$.  If $a>t_*^{1/4}$, then among all triangles with area 1 and a fixed side of length $a$, the isosceles triangle with base $a$ does not maximize $\rho_2$.
\end{theorem}

\begin{proof}
Let $\hat{\Omega}$ be an area-normalized triangle with fixed side length $a$.
 
We begin by proving part (i).  As $\rho_1$ is rotationally invariant, we may position $\hat{\Omega}$ so that the side of fixed length is parallel to the $y$-axis as seen in Figure 2. Denote vertex $\hat{A}$ as the origin, $\hat{B}$ as the point $(0,a)$, and $\hat{C}$ as the point $(-2/a,\lambda)$.
    {
\begin{figure}[H] 
    \centering
    \includegraphics[scale=0.75]{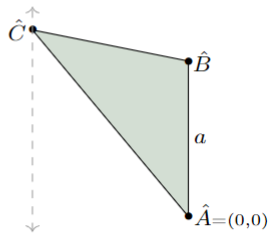}
    \caption{\textit{An area-normalized triangle $\hat{\Omega}$ with variable $\lambda$ and fixed base length $a$.}}
    \label{label4}
\end{figure}}
    
Notice as $\lambda$ varies, the vertex $\hat{C}$ stays on the line $x=-\frac{2}{a}$ in order to preserve area-normalization. If we define
\begin{equation}\label{txty}
T_x:=\int_{\hat{\Omega}}xdA\qquad\qquad\mbox{and}\qquad\qquad T_y:=\int_{\hat{\Omega}}ydA
\end{equation}
then we may translate our triangle to obtain a new triangle $\Omega$ with vertices $A$, $B$, and $C$ given by
\begin{align*}
A&=\left(-T_x,-T_y\right)\\
B&=\left(-T_x,a-T_y\right)\\
C&=\left(-\frac{2}{a}-T_x,\lambda-T_y\right)
\end{align*}
which has centroid zero (see Figure 3).

    {    
    \begin{figure}[H] 
    \centering
    \includegraphics[scale=0.75]{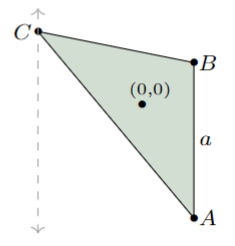}
    \caption{\textit{The area-normalized triangle $\Omega$ as pictured is a translation of $\hat\Omega$, with variable $\lambda$, fixed base length $a$, and whose centroid is the origin.}}
    \label{label5}
\end{figure}}

If we define
\begin{align*}
\ell_1&=\lambda-\frac{\lambda a}{2}\left(x+\frac{2}{a}\right),\qquad\qquad\ell_2=\lambda+\frac{a^2-a\lambda}{2}\left(x+\frac{2}{a}\right),
\end{align*}
by recalling our formula for the moments of area, we have
\[
I_{m,n}(\Omega)=\int_{-\frac{2}{a}-T_x}^{-T_x}\int_{\ell_1}^{\ell_2}x^my^ndydx
\]
We can now calculate $\rho_1$ using \eqref{rho1form} to obtain
\begin{equation}
\rho_1(\Omega)=\frac{2a^2}{3(4+a^2(a^2-a\lambda+\lambda^2))}
\end{equation}
By taking the first derivative with respect to $\lambda$ of equation (5) we obtain
\[
\frac{d}{d\lambda}\left[\rho_1(\Omega)\right]=\frac{2a^4(a-2\lambda)}{3(4+a^2(a^2-a\lambda+\lambda^2))^2}
\]
Thus, the only critical point is when $\lambda=\frac{a}{2}$, when the $y$-coordinate of the vertex $\hat{C}$ is at the midpoint of our base.

To prove part (ii), we employ the same method, but use the formula from Proposition \ref{rho2def}.  After a lengthy calculation, we find a formula
\[
\frac{d}{d\lambda}\left[\rho_2(\Omega)\right]=\frac{P(\lambda)}{Q(\lambda)}
\]
for explicit polynomials $P$ and $Q$.  The polynomial $Q$ is always positive, so we will ignore that when finding critical points.  By inspection, we find that we can write
\[
P(\lambda)=(\lambda-a/2)S(\lambda),
\]
where $S(\lambda+a/2)$ is an even polynomial.  Again by inspection, we find that every coefficient of $S(\lambda+a/2)$ is negative, except the constant term, which is
\[
999a^{20}/64-93a^{16}-664a^{12}-5376a^8-9216a^4.
\]
Thus, if $0<a<t_*^{1/4}$, then this coefficient is also negative and therefore $S(\lambda+a/2)$ does not have any positive zeros (and therefore does not have any real zeros since it is an even function of $\lambda$).  If $a>t_*^{1/4}$, then $S(\lambda+a/2)$ does have a unique positive zero and it is easy to see that it is a local maximum of $\rho_2$ (and the zero of $P$ at $a/2$ is a local minimum of $\rho_2$).
\end{proof}

We remark that the number $t_*^{1/4}$ from Theorem \ref{steinertri} is approximately $1.86637\ldots$ and $\sqrt{3}=1.73205\ldots$, so Theorem \ref{steinertri} does not disprove Conjecture \ref{rhon}.  In Figure 4, we have plotted $\rho_2$ as a function of $\lambda$ when $a=1$.  We see the maximum is attained when $\lambda=1/2$.  In Figure 5, we have plotted $\rho_2$ as a function of $\lambda$ when $a=3$.  We see that $\lambda=3/2$ is a local minimum and the maximum is actually attained when $\lambda=3/2\pm0.86508\ldots$. 

    {    
    \begin{figure}[H] 
    \centering
    \includegraphics[scale=2.7]{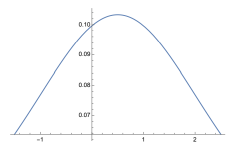}
    \caption{\textit{$\rho_2$ of a triangle with area $1$ and fixed side length $1$.}}
    \label{rho2-1.pdf}
\end{figure}}

    {    
    \begin{figure}[H] 
    \centering
    \includegraphics[scale=2.7]{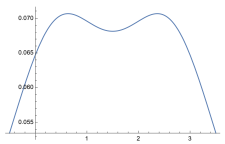}
    \caption{\textit{$\rho_2$ of a triangle with area $1$ and fixed side length $3$.}}
    \label{label7}
\end{figure}}


We can now prove the following corollary, which should be interpreted in contrast to Corollary \ref{nosteiner}.

\par\begin{corollary}\label{equimax1}
Given an arbitrary triangle of area 1, Steiner symmetrization performed parallel to one of the sides increases $\rho_1$. Consequently, the equilateral triangle is the unique maximum for $\rho_1$ among all triangles of fixed area. 
\end{corollary}

\begin{proof}
We saw in the proof of Theorem \ref{steinertri} that if a triangle has any two sides not equal, than we may transform it in a way that increases $\rho_1$.  The desired result now follows from the existence of a triangle that maximizes $\rho_1$.
\end{proof}

We can also consider a related problem of maximizing $\rho_N$ among all triangles with one fixed angle.  To this end, we formulate the following conjecture, which is analogous to Problem \ref{fixedbasen}.  If true, it would be an analog of results in \cite{Sol20} for the Bergman $N$-polynomial content.

\begin{conj}\label{fixedanglen}
For an $N\in\bbN$ and $\theta\in(0,\pi)$, the triangle with area $1$ and fixed interior angle $\theta$ that maximizes $\rho_N$ is the isosceles triangle with area $1$ and interior angle $\theta$ opposite the base.
\end{conj}

The following theorem provides strong evidence that Conjecture \ref{fixedanglen} is true.

\begin{theorem}\label{fixedangle}
Among all triangles with area 1 and fixed interior angle $\theta$, the isosceles triangle with interior angle $\theta$ opposite the base maximizes $\rho_1$ and $\rho_2$.
\end{theorem}

\begin{proof}
Let $\Omega$ be an area-normalized triangle with fixed interior angle $\theta$, centroid zero, and side length $a$ adjacent to our angle $\theta$. As $\rho_N$ is rotationally invariant, let us position $\Omega$ so that the side of length $a$ runs parallel to the $x$-axis. First, let us consider the triangle $\hat\Omega$ which is a translation of $\Omega$, so that the corner of $\hat\Omega$ with angle $\theta$, say vertex $A$, lies at the origin. Define $(T_x,T_y)$ as in \eqref{txty}. By translating the entire region $\hat\Omega$ by its centroid we attain the previously described region $\Omega$, now with centroid zero, as pictured in Figure 6.

{
\begin{figure}[H]
    \centering
    \includegraphics[scale=0.75]{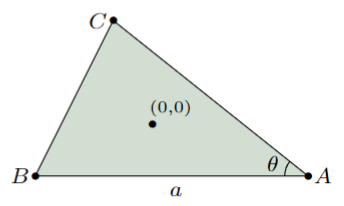}
    \caption{\textit{A triangle $\Omega$ with fixed angle $\theta$, variable side length $a$, area 1, and centroid zero.}}
    \label{Triangle4.png}
\end{figure}}

Our region $\Omega$ is now a triangle with centroid $0$ having vertices $A$, $B$, and $C$ given by
\begin{align*}
    A&=(-T_x,-T_y)\\
    B&=(-a-T_x,-T_y)\\
    C&=\left(\frac{-2}{a\tan\theta}-T_x,\frac{2}{a}-T_y\right)
\end{align*}
We can now use \eqref{rho1form} to calculate
\begin{equation}
\rho_1(\Omega)=\frac{2a^2}{3a^4-6a^2\cot\theta+12\csc^2\theta}
\end{equation}

By taking the first derivative with respect to $a$ of equation (4) we obtain 
\[
\frac{d}{da}\left[{\rho_1(\Omega)}\right]=\frac{-4a(a^4-4\csc^2\theta)}{3(a^4-2a^2\cot\theta+4\csc^2\theta)^2}
\] 
Thus, the only critical point of $\rho_1$ is $a=\sqrt{2\csc\theta}$ and this point is a local maximum. We conclude our proof by observing that $a=\sqrt{2\csc\theta}$ is the side length of the area-normalized isosceles triangle with interior angle $\theta$ opposite the base.

The calculation for $\rho_2$ follows the same basic strategy, albeit handling lengthier calculations.  In this case, we calculate
\[
\frac{d}{da}\left[{\rho_2(\Omega)}\right]=\frac{Q_{\theta}(a)}{P_{\theta}(a)}
\]
for explicitly computable functions $Q_{\theta}$ and $P_{\theta}$, which are polynomials in $a$ and have coefficients that depend on $\theta$.  The function $P_{\theta}(a)$ is positive for all $a>0$, so the zeros will be the zeros of $Q_{\theta}$.  One can see by inspection that $Q_{\theta}(\sqrt{2\csc\theta})=0$, so let us consider
\[
S_{\theta}(a)=Q_{\theta}(a+\sqrt{2\csc\theta}).
\]
Then $S(0)=0$ and there is an obvious symmetry to these triangles that shows the remaining real zeros of $S$ must come in pairs with a positive zero corresponding to a negative one.  Thus, it suffices to rule out any positive zeros of $S_{\theta}$.  This is done with Descartes' Rule of Signs, once we notice that all coefficients of $S_{\theta}$ are negative.  For example, one finds that the coefficient of $a^7$ in $S_{\theta}(a)$ is equal to
\[
-12288\csc^7\theta(140\cos(4\theta)-3217\cos(3\theta)+25010\cos(2\theta)-82016\cos(\theta)+70136)
\]
One can plot this function to verify that it is indeed negative for all $\theta\in[0,\pi]$.  Similar elementary calculations can be done with all the other coefficients in the formula for $S_{\theta}(a)$, but they are too numerous and lengthy to present here.

The end result is the conclusion that $a=\sqrt{2\csc\theta}$ is the unique positive critical point of $\rho_2$ and hence must be the global maximum, as desired.
 \end{proof}
 
 We can now prove the following result, which is the a natural follow-up to Corollary \ref{equimax1}.
 
 \par\begin{corollary}\label{equimax2}
The equilateral triangle is the unique maximum for $\rho_2$ among all triangles of fixed area. 
\end{corollary}

\begin{proof}
We saw in the proof of Theorem \ref{fixedangle} that if a triangle has any two sides not equal, than we may transform it in a way that increases $\rho_2$.  The desired result now follows from the existence of a triangle that maximizes $\rho_2$.
\end{proof}

The same proof shows that Corollary \ref{equimax1} is also a corollary of Theorem \ref{fixedangle}.

\section{Numerics on Torsional Rigidity for Pentagons}\label{pentnum}

Here we present numerical evidence in support of P\'olya's conjecture for pentagons.  In particular, we will consider only equilateral pentagons and show that in this class, the maximizer of torsional rigidity must be very close to the regular pentagon (see \cite{CCGINPT21} for another computational approach to a similar problem).  Our first task is to show that to every $\theta,\phi\in(0,\pi)$ satisfying
\begin{align*}
(1-\cos(\theta)-\cos(\phi))^2+&(\sin(\theta)-\sin(\phi))^2\leq 4,\\
\cos(\theta)&\leq1-\cos(\phi),
\end{align*}
there exists a unique equilateral pentagon of area $1$ with adjacent interior angles $\theta$ and $\phi$ (where the uniqueness is interpreted modulo rotation, translation, and reflection).

{\begin{figure}[H]
    \centering
    \includegraphics[scale=0.75]{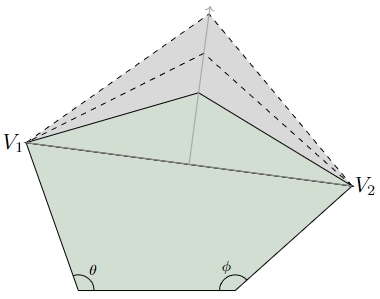}
    \caption{\textit{A pentagon constructed with vertices $V_1$, $V_2$, and interior angles $\theta$, $\phi$ as described below. There is exactly one point on the perpendicular bisector of $\overline{V_1V_2}$ for which our pentagon is equilateral.}}
    \label{Pentagon.png}
\end{figure}}

To see this, construct a pentagon with one side being the interval $[0,1]$ in the real axis.  Form two adjacent sides of length $1$ with interior angles $\phi$ and $\theta$ by choosing vertices $V_1=(\cos(\theta),\sin(\theta))$ and $V_2=(1-\cos(\phi),\sin(\phi))$.  Our conditions imply that $V_1$ lies to the left of $V_2$ and the distance between $V_1$ and $V_2$ is less than or equal to $2$.  Thus, if we join each of $V_1$ and $V_2$ to an appropriate point on the perpendicular bisector of the segment $\overline{V_1V_2}$, we complete our equilateral pentagon with adjacent angles $\theta$ and $\phi$ (see Figure 7).  Obtaining the desired area is now just a matter of rescaling.

Using this construction, one can write down the coordinates of all five vertices, which are simple (but lengthy) formulas involving basic trigonometric functions in $\theta$ and $\phi$.  It is then a simple matter to compute a double integral and calculate the area of the resulting pentagon, rescale by the appropriate factor and thus obtain an equilateral pentagon with area $1$ and the desired adjacent internal angles.  One can then compute $\rho_N$ for arbitrary $N\in\bbN$ using the method of \cite{FS19} to estimate the torsional rigidity of such a pentagon.

Theoretically, this is quite simple, but in practice this is a lengthy calculation.  We were able to compute $\rho_{33}(\Omega)$ for a large collection of equilateral pentagons $\Omega$.  Note that all interior angles in the regular pentagon are equal to $108$ degrees.  We discretized the region $\theta,\phi\in[105,110]$ (in degrees) and calculated $\rho_{33}$ for each pentagon in this discretization.  The results showed a clear peak near $(\theta,\phi)=(108,108)$, so we further discretized the region $\theta,\phi\in[107.5,108.5]$ (in degrees) into $400$ equally spaced grid points and computed $\rho_{33}$ for each of the $400$ pentagons in our discretization.  We then interpolated the results linearly and the resulting plot is shown as the orange surface in Figure 8.

{
\begin{figure}[H] 
    \centering
    \includegraphics[scale=0.75]{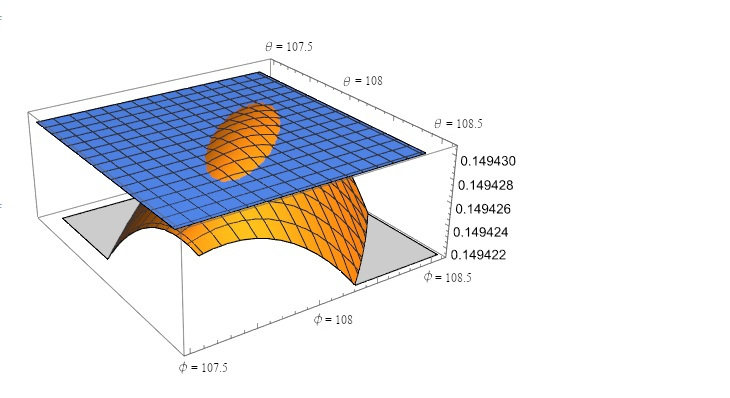}
    \caption{\textit{$\rho_{33}$ for a selection of equilateral pentagons with area $1$ having angles close to those of the regular pentagon.}}
    \label{2a.jpg}
\end{figure}}

The blue surface in Figure 8 is the plane at height 0.149429, which is the (approximate) torsional rigidity of the area normalized regular pentagon calculated by Keady in \cite{Keady}.  Recall that every $\rho_N$ is an overestimate of $\rho$, so any values of $\theta$ and $\phi$ for which $\rho_{33}$ lies below this plane will not be the pentagon that maximizes $\rho$.  Thus, if we take the value 1.49429 from \cite{Keady} as the exact value of the torsional rigidity of the regular pentagon with area $1$, we see that among all equilateral pentagons, the maximizer of $\rho$ will need to have two adjacent angles within approximately one third of one degree of $108$ degrees.  This is extremely close to the regular pentagon, and of course the conjecture is that the regular pentagon is the maximizer.

\bigskip

\noindent\textbf{Acknowledgements.}  The second author graciously acknowledges support from the Simons Foundation through collaboration grant 707882.

\end{document}